\documentclass[aos,preprint]{imsart}
\RequirePackage[OT1]{fontenc}
\RequirePackage{amsthm,amsmath}
\RequirePackage[numbers]{natbib}
\usepackage{pictexwd}
\usepackage{bm,amssymb,graphicx,url}
\usepackage{color}
\usepackage{epstopdf}
\usepackage[normalem]{ulem}

\usepackage[colorlinks,citecolor=blue,urlcolor=blue]{hyperref}

\startlocaldefs    
\numberwithin{equation}{section}
\theoremstyle{plain}
\endlocaldefs

\usepackage{amsmath,here,amsfonts,latexsym,graphicx,here}
\usepackage{amsthm,curves}
\usepackage{graphicx}
\usepackage{caption}
\usepackage{subcaption}
\usepackage{float}

\def\m{\mu}
\def\th{\theta}

\def\P{{\mathbb P}}
\def\Q{{\mathbb Q}}

\def\m{\mu}

\def\th{\theta}

\def\P{{\mathbb P}}

\def\s{\sigma}
\def\R{\mathbb R}
\def\dd{\Delta}
\def\d{\delta}
\def\f{\phi}
\def\IK{I\!\!K}

\newtheorem{theorem}{Theorem}[section]
\newtheorem{lemma}{Lemma}[section]

\newtheorem{remark}{Remark}[section]

\newtheorem{example}{Example}[section]
\newtheorem{definition}{Definition}[section]
\newtheorem{conjecture}{Conjecture}[section]

\begin{document}
\begin{frontmatter}
\title{Nonparametric Estimation in Uniform Deconvolution and Interval Censoring}
\runtitle{Uniform deconvolution}

\begin{aug}
\author{\fnms{Piet} \snm{Groeneboom}\corref{}\ead[label=e1]{P.Groeneboom@tudelft.nl}}
\and
\author{\fnms{Geurt} \snm{Jongbloed}\corref{}\ead[label=e2]{G.Jongbloed@tudelft.nl}}
\runauthor{Piet Groeneboom and Geurt Jongbloed}
\address{Delft Institute of Applied Mathematics, Mekelweg 4, 2628 CD Delft,
	The Netherlands.\\ 
	\printead{e1,e2}} 
\end{aug}

\begin{abstract}
In the uniform deconvolution problem one is interested in estimating the distribution function $F_0$ of a nonnegative random variable, based on a sample with additive uniform noise. A peculiar and not well understood phenomenon of the nonparametric maximum likelihood estimator in this setting is the dichotomy between the situations where $F_0(1)=1$ and $F_0(1)<1$. If $F_0(1)=1$, the MLE can be computed in a straightforward way and its asymptotic pointwise behavior can be derived using the connection to the so-called current status problem. However, if $F_0(1)<1$, one needs an iterative procedure to compute it and the asymptotic pointwise behavior of the nonparametric maximum likelihood estimator is not known. In this paper we describe the problem, connect it to interval censoring problems and a more general model studied in \cite{piet_EJS:24} to state two competing naturally occurring conjectures for the case $F_0(1)<1$. Asymptotic arguments related to smooth functional theory and extensive simulations lead us to to bet on one of these two conjectures.
\end{abstract} 



\end{frontmatter}
\section{Introduction}
\label{sec:intro}
Consider a random variable $U$, distributed according to a distribution function $F_0$ on $\R^+$. Instead of observing $U$, one observes the sum
$$
S=U+V 
$$
where $V\sim$Unif$(0,1)$, independent of $U$. The random variable $S$ then has convolution density
\begin{align}
\label{eq:directrelcur}
g_S(s)=\int 1_{[0,1]}(s-v)\,dF_0(v)=\int 1_{[s-1,s]}(v)\,dF_0(v)=F_0(s)-F_0(s-1).
\end{align}
Statistical inference for $F_0$ based on $n$ i.i.d.\ random variables $S_1,\ldots,S_n$ with density $g_S$, is known as the uniform deconvolution problem. It is not essential that the support of the uniform distribution is $[0,1]$, we can always reduce the deconvolution problem with uniform random variables with another support to this situation. For simplicity we stick in this paper to uniform random variables with support $[0,1]$.  As discussed in \cite{GroJoUnifDec:03}, an interesting application of uniform deconvolution is in the ``deblurring" of pictures, blurred by Poisson noise, see, e.g., \cite{kingshuk:01}.

The log likelihood of $F$ is then defined as
\begin{align}
\label{loglike1}
\ell(F)=\sum_{i=1}^n\log\left\{F(S_i)-F(S_i-1)\right\}.
\end{align}
If $F_0$ satisfies $F_0(1)=1$, this model will be seen to be equivalent to the so-called current status (or interval censoring case 1) model. Consequently, there is  a one step algorithm to compute the nonparametric MLE of $F_0$. We explain this further in Section \ref{subsec:support_in_[0,1]}. We also derive the asymptotic distribution of the nonparametric MLE in this section.

In section \ref{subsec:support_not_in_[0,1]}, we show that if the support of the distribution corresponding to the distribution function $F_0$ is not contained in $[0,1]$, with upper support point $M>1$, the model can be shown to be equivalent to an interval censoring, case $m$, model, where $m=\lceil M\rceil$. 
Consequently, a one step algorithm to compute the nonparametric MLE is not known, but one can use iterative algorithms in this case to compute the nonparametric MLE, for example the iterative convex minorant algorithm, proposed in \cite{GrWe:92}.

A more general model is studied in \cite{piet_EJS:24} and \cite{piet_Statistica:24}. There the length of the support of the uniform random variable $V$, say $E$, is actually random (but observable).
This model is used for estimating the distribution of the incubation time of a disease. In this interpretation there is a positive random variable $E$, the duration of the ``exposure time", with (unknown) distribution function $F_E$. Given $E$, an independent ``infection time" $V$ is drawn, uniformly distributed on the interval $[0,E]$. Moreover, the quantity of interest is the incubation (length) time $U$ with distribution function $F_0$. Apart from $E$, the observation is then the time $S$ for getting symptomatic, where $S=U+V$. Here, the random variables $U$ and $V$ are independent, conditionally on $E$.  The observations therefore consist of $n$ i.i.d.\ pairs of exposure times and times of getting symptomatic
\begin{align}
\label{observations_incub}
(E_i,S_i),\qquad i=1,\dots,n.
\end{align}
This model is also considered in \cite{backer:20},  \cite{tom_gianpi:19}, \cite{piet:21} and \cite{reich:09}.

Under the assumption that $E$ has an absolutely continuous distribution, the asymptotic pointwise distribution of the nonparametric MLE of $F_0$ is derived in \cite{piet_EJS:24}. Remarkably, for this result there is no distinction between different assumptions on the support of the distribution corresponding to $F_0$, it holds in general.  Taking for the distribution of $E$ the point mass at $1$ in the statement of Theorem 4.1 in \cite{piet_EJS:24}, a case certainly not covered by the assumption of absolute continuity in this theorem and coinciding with the uniform deconvolution problem stated above, the correct asymptotic distribution appears in case $F_0(1)=1$ as derived in our section \ref{subsec:support_in_[0,1]}. This also provides a natural conjecture (which we do, however, not believe in) for the asymptotic distribution of the nonparametric MLE in the situation where $F_0(1)<1$. In section  \ref{sec:mixedmodel} we discuss this more general model, used in estimating the incubation time distribution, in more detail and formulate the conjecture.

Before turning to a simulation study and asymptotic considerations to support our actual conjecture in section \ref{sec:asymptandsimul},  in section \ref{sec:examples} we digress in a discussion of the estimation of other functionals of $F_0$, rather than the evaluation of the nonparametric MLE itself at a fixed point. For these so-called smooth functionals, asymptotic results are known. Moreover, the idea behind these smooth functionals is needed to understand the asymptotic considerations leading up to our final conjecture for the asymptotic distribution of the nonparametric MLE of $F_0(t_0)$ in the uniform deconvolution setting.

\section{Uniform deconvolution and interval censoring}
\label{sec:asymptotic_uniform_deconvolution}
In this section, we establish a relation between the uniform deconvolution problem and various interval censoring models. In those models, also a dichotomy between two cases arises naturally. The uniform deconvolution problem in case $F_0(1)=1$ will be seen to be related to the interval censoring case 1 (or, current status) problem. This situation will be discussed in section  \ref{subsec:support_in_[0,1]}, where also the asymptotic distribution of the nonparametric MLE of $F_0(t_0)$ is given. In section \ref{subsec:support_not_in_[0,1]}, we discuss the relation between the uniform deconvolution model where $F_0(1)<1$ and interval censoring case $m$, for $m\ge 2$.

We first briefly describe the  interval censoring case 1 (IC-$1$) problem. Consider a random variable $X$ with distribution function $F_0$ of interest and independently of this a random variable $Y$  with a distribution function $G$. Instead of observing $X$, the pair $(Y,\Delta)$ is observed, where $\Delta$ indicates whether $X$ is smaller than or equal to $Y$. So, $\Delta=1_{\{X\le Y\}}$. In survival analysis terminology, $X$ is the event time, $Y$ is the inspection time and we observe the ``status'' of the subject at time $Y$. The status being whether the event (usually an infection) has already occured at time $Y$ or not. Note that given inspection time $Y=y$, $\Delta\sim$Bernoulli$(F_0(y))$. The problem is to estimate  distribution function $F_0$, based on an i.i.d.\ sample of $(Y_i,\Delta_i)$'s.

A more general model is the interval censoring case $m$ (IC-$m$) model, where $m$ stands for the number of distinct inspection times per subject. Formally, $Y=(Y_1,\ldots,Y_m)$ is a random vector of distinct inspection times and $X$ a nonnegative random variable independent of these event times. One then observes $(Y,\Delta)$, where the $m+1$-vector $\Delta$ indicates  which of the intervals defined by the vector $Y$ contains $X$: $\Delta=(\Delta_1,\ldots,\Delta_{m+1})$. Here for $1\le j\le m+1$, $\Delta_j=1_{(Y_{(j-1)},Y_{(j)}]}(X)$, denoting by $Y_{(j)}$ the $j$-th order statistic in the vector $Y$,  and by convention $Y_0\equiv 0$ and  $Y_{(m+1)}=\infty$. Note that given the (ordered) vector of observation times equals $(y_{(1)},\ldots,y_{(m)})$, the vector $\Delta$ is multinomially distributed with parameters $1$ and probability vector
$$
(F_0(y_{(1)}),F_0(y_{(2)})-F_0(y_{(1)}),\ldots,F_0(y_{(m)})-F_0(y_{(m-1)}),1-F_0(y_{(m)})).
$$
Given an i.i.d. sample of vectors $(Y,\Delta)$, the problem is then to estimate the distribution function $F_0$.

In the two subsections below, we establish the relation between the uniform deconvolution problem and the interval censoring models and derive results based on that, depending on whether $F_0(1)=1$ or $F_0(1)<1$.

\subsection{Case $F_0(1)=1$}
\label{subsec:support_in_[0,1]}
Following the suggestions of Exercise 2, p.\,61, of \cite{GrWe:92}, we
set up an equivalence of the interval censoring problem with the current status model, in case $F_0(1)=1$. Given the observed $S_1,S_2,\ldots,S_n$, define
\begin{align}
\label{delta}
\dd_i=\left\{\begin{array}{ll}
1 &,\text{ if }S_i\le1\\
0&,\text{ if }S_i>1
\end{array}
\right.
\end{align}
and let the random variable $Y_i$ be defined by:
\begin{align}
\label{def_Y_i}
Y_i=\left\{\begin{array}{ll}
S_i, &,\text{ if }\dd_i=1\\
S_i-1&,\text{ if }\dd_i=0,
\end{array}
\right.
\end{align}
translating the $S_i$'s to pairs $(Y_i,\Delta_i)$ for $1\le i\le n$.
Note that, for $y\in(0,1)$, considering a single observation (e.g. $i=1$), we get
\begin{align*}
\P(Y\le y)&=\P(Y\le y \wedge \Delta=0)+\P(Y\le y \wedge \Delta=1)\\
&=\P(S-1\le y \wedge \Delta=0)+\P(S\le y \wedge \Delta=1)
=\P(1<S\le y+1)+\P(S\le y)\\
&=\int_1^{1+y}g_S(s)\,ds+\int_0^{y}g_S(s)\,ds=\int_0^yF_0(s+1)\,ds=y.
\end{align*}
Here we use (\ref{eq:directrelcur}) and the fact that $F_0(s+1)=1$ if $s>0$, as $F_0(1)=1$.
Hence, the $Y_i$'s are uniformly distributed on $[0,1]$.

Moreover, again using (\ref{eq:directrelcur}), for $y\in(0,1)$, the conditional distribution of $\Delta$ given $Y=y$ is Bernoulli, with success parameter
\begin{align*}
\P\left\{\dd=1|Y=y\right\}=\frac{g_S(y)}{g_S(y)+g_S(y+1)}=F_0(y).
\end{align*}

This shows the equivalence of the uniform deconvolution model with the current status model in the case $F_0(1)=1$. More specifically, it is a current status model with event time distrubution $F_0$ and inspection time distribution uniform on $[0,1]$. This means that we can immediately characterize the nonparametric MLE in this setting and also derive its pointwise asymptotic distribution. Indeed, the MLE is the maximizer of the log likelihood
\begin{align*}
\sum_{i=1}^n\left(\dd_i\log F(Y_i)+(1-\dd_i)\log(1-F(Y_i))\right),
\end{align*}
where $\dd_i$ and $Y_i$ are defined by (\ref{delta}) and (\ref{def_Y_i}) as functions of the observations $S_i$. Being a special instance of a `generalized isotonic regression problem', the MLE $\hat F_n$ is known to coincide with the solution of the isotonic regression problem of minimizing the sum of squares
\begin{align*}
\sum_{i=1}^n\left(\dd_i- F(Y_i)\right)^2,
\end{align*}
over all distribution functions $F$. The one step algorithm for computing the solution, using the so-called cusum diagram of the $\dd_i$'s is described, e.g., on p. 30 of \cite{piet_geurt:14}. The solution consists of local averages of the $\dd_i$ and can therefore only have rational values.

Specializing Theorem 3.7 in Section 3.8 of \cite{piet_geurt:14} to the setting where the inspection times are uniformly distributed on $[0,1]$, yields the asymptotic distribution of the MLE in the uniform deconvolution model with $F_0(1)=1$.

\begin{theorem}
\label{th:local_limit}
Let $F_0$ be differentiable on $(a,b),\,0<a<b<1$ with a continuous positive derivative $f_0(t)$ for $t\in(a,b)$, where $[a,b]$ is the support of $f_0$. Let $\hat F_n$ be the nonparametric MLE of $F_0$. Then, for $t_0\in(a,b)$:
\begin{align}
\label{local_limit_result}
n^{1/3}\{\hat F_n(t_0)-F_0(t_0)\}/\left(4f_0(t_0)F_0(t_0)(1-F_0(t_0))\}\right)^{1/3}\stackrel{d}\longrightarrow \text{\rm argmin}_{t\in\R}\left\{W(t)+t^2\right\},
\end{align}
where $W$ is two-sided Brownian motion on $\R$, originating from zero.
\end{theorem}

\subsection{Case $F_0(1)<1$}
\label{subsec:support_not_in_[0,1]}
In line with the previous subsection, we will make a connection of the uniform deconvolution problem to the more general IC-$m$ problem, where $m=\lceil M\rceil$, $M$ denoting the upper support point of the distribution corresponding to distribution function $F_0$. 

We start again by constructing a vector $\dd$ and $Y$ based on the observation $S$. Note that $S$ now takes values in $[0,M+1]$. Define for $1\le j\le m+1$ the $j$-th component in $\Delta$ by
\begin{align}
\label{deltavec}
\dd_j=\left\{\begin{array}{ll}
1 &,\text{ if }S\in(j-1,j]\\
0&,\text{ else. }
\end{array}
\right.
\end{align}
Moreover, define the components of the vector $Y$ by
\begin{align}
\label{def_Yvec_j}
Y_j=
S-\lfloor S\rfloor+j-1
\end{align}
for $1\le j\le m+1$.

Similarly to the IC-$1$ setting, we can show that  $Y_1$ is uniformly distributed on $[0,1]$. Indeed, for $y\in(0,1)$,
\begin{align*}
\P(Y_1\le y)&=\sum_{j=1}^{m+1}\P(Y_1\le y \wedge \Delta_j=1)=\sum_{j=1}^{m+1}\P(Y_1\le y \wedge j-1<S\le j)\\
&=\sum_{j=1}^{m+1}\P(S-(j-1)\le y \wedge j-1<S\le j)=\sum_{j=1}^{m+1}\P(j-1<S\le y+j-1)\\
&=\sum_{j=1}^{m+1}\int_{j-1}^{j-1+y}g_S(s)\,ds=\sum_{j=1}^{m+1}\int_{j-1}^{j-1+y}F_0(s)-F_0(s-1)\,ds=\int_{m}^{m+y}F_0(s)\,ds
\end{align*}
Again, the right hand side equals $y$ because $F_0(s)=1$ for $s\ge m$, as $m\ge M$ and $F_0(M)=1$. So, the vector $Y$ consists of a uniformly distributed random variable $Y_1$ on $[0,1]$, followed by $Y_j=Y_1+j-1$ for $2\le j\le m$. It is therefore also automatically ordered.

Now, consider the conditional distribution of $\Delta$ given $Y$, which only depends on $Y_1$. For fixed $j\in\{1,2,\ldots,m+1\}$ and $y_1\in(0,1)$, using (\ref{eq:directrelcur})
\begin{align*}
\P(\Delta_j=1|Y_1=y_1)&=\frac{g_S(j-1+y_1)}{\sum_{k=1}^{m+1}g_S(k-1+y_1)}\\&=\frac{F_0(j-1+y_1)-F_0(j-2+y_1)}{F_0(m+y_1)-F_0(y_1-1)}=F_0(y_{(j)})-F_0(y_{(j-1)}),
\end{align*}
because $y_1-1<0$ and $m+y_1>m\ge M$.
In view of the previous discussion on the IC-$m$ model, this shows that translating the sample of $S_i$'s to the sample of $(Y_i,\Delta_i)$'s according to (\ref{deltavec}) and (\ref{def_Yvec_j}), yields IC-$m$ data. The distribution of `inspection times' is quite specific and peculiar, having distance exactly equal to one with a uniformly distributed first.

Having established this relation between the two models, results known for the IC-$m$ model can be used to derive results for our deconvolution model. However, much less results are known for the IC-$m$ model than for IC-$1$. One indication of the difficulties in going away from the IC-1 model  is that the nonparametric MLE can have irrational values in the IC-$m$, $m>1$, model. A simple example of this is given in Example 1.3 on p.\ 48 of \cite{GrWe:92}. Here the values of the nonparametric maximum likelihood estimate of the distribution function of the hidden variable contain the values $\tfrac12\pm\tfrac16\sqrt{3}$ (the solution can be computed analytically in this simple example). This cannot happen in the IC-$1$ model. 

Let us now turn to the nonparametric MLE. Define
\begin{align}
\label{def_m_n}
m_n=\max_{j:S_j>1}(S_j-1).
\end{align} 
Denoting by $\Q_n$  the empirical distribution of $(S_1,\dots,S_n)$, the nonparametric MLE is  defined as maximizer of the log likelihood function
\begin{align}
\label{emp_loglik}
\ell_n(F)&=\int\log\{F(s)-F(s-1)\}\,d\Q_n(s)\nonumber\\
&=\int_{s\le 1}\log F(s)\,d\Q_n(s)+\int_{1<s\le m_n}\log\{F(s)-F(s-1)\}\,d\Q_n(s)\nonumber\\
&\quad\qquad\quad\qquad\quad\qquad\quad\qquad+\int_{s>1\vee m_n}\log\{1-F(s-1)\}\,d\Q_n(s).
\end{align}
Note that for maximizing this log likelihood function, the class of distribution functions can be restricted to those satisfying $F(S_i)=1$ if $S_i>\max_{j:S_j>1}(S_j-1)$, and $F(S_i-1)=0$ if $S_i-1<\min_j S_j$, for different values of $F$ at these points would make the log likelihood smaller.
Define the set of distribution functions.
\begin{definition}
\label{def_cal_F_n}
{\rm
${\cal F}_n$ is the set of discrete distribution functions $F$, which only have mass at the points $\{S_i,S_i-1\,:\,1\le i\le n\}$  and satisfy 
\begin{align}
\label{upper_values}
F(x)=1\quad \text{\rm if }\quad x\ge m_n,
\end{align}
and
\begin{align}
\label{lower_values}
F(x)=0\quad \text{\rm if }\quad x<\min_jS_j.
\end{align}
}
\end{definition} 
We restrict the distribution functions, occurring in the problem of maximizing the likelihood to this set ${\cal F}_n$.

The nonparametric MLE can then be characterized using the following process, defined in terms of $F\in {\cal F}_n$
\begin{align}
\label{process_W}
&W_{n,F}(t)\nonumber\\
&=\int\frac{\{s\le t\wedge 1\}}{F(s)}\,d\Q_n(s)-\int\frac{\{0<s-1\le t<s\le m_n\}}{F(s)-F(s-1)}\,d\Q_n(s)\nonumber\\
&\quad+\int\frac{\{1<s\le t\wedge m_n\}}{F(s)-F(s-1)}\,d\Q_n(s)-\int\frac{\{1\vee m_n-1<s-1\le t\}}{1-F(s-1)}\,d\Q_n(s),\quad t\ge0,
\end{align}
where we write the set as shorthand for its indicator function.

\vspace{1cm}
Now, let $T_1< \dots < T_m$ be the points $S_i$ such that $S_i$ not of type (\ref{upper_values}) and $S_i-1$ is not of type (\ref{lower_values}). Then $0<F(T_i)<1$ for $i=1,\dots,m$, if the log likelihood for $F$ is finite, and we can define:
\begin{align*}
{\cal Y}=\{\bm y\in(0,1)^m:\bm y=(y_1,\dots,y_m)=(F(T_1),\dots,F(T_m),\,F\in{\cal F}_n\}.
\end{align*}

The following lemma characterizes the MLE and is proved in a completely analogous way as Lemma 2.2 in \cite{piet_EJS:24}. 
\begin{lemma}
\label{lemma:char_MLE}
Let the class of distribution functions ${\cal F}_n$ be as defined in Definition \ref{def_cal_F_n}. Then $\hat F_n\in{\cal F}_n$ maximizes (\ref{emp_loglik}) over $F\in{\cal F}_n$ if and only if
\begin{enumerate}
\item[(i)] 
\begin{align}
\label{fenchel1}
\int_{u\in[t,\infty)}\,dW_{n,\hat F_n}(u)\le0,\qquad t\ge0,
\end{align}
\item[(ii)]
\begin{align}
\label{fenchel2}
\int \hat F_n(t)\,dW_{n,\hat F_n}(t)=0.
\end{align}
\end{enumerate}
where $W_{n,F_n}$ is defined by (\ref{process_W}). Moreover, $\hat F_n\in{\cal F}_n$ is uniquely determined by (\ref{fenchel1}) and (\ref{fenchel2}). 
\end{lemma}

As mentioned above, in order to compute the MLE in this case,  an iterative algorithm has to be used. We use the iterative convex minorant algorithm proposed in \cite{GrWe:92}, using line search to determine the step size, as described in \cite{geurt:98}. This algorithm was also applied in computing the MLE for the incubation time distribution in \cite{piet:21}. {\tt R} scripts for computing the MLE in this way are available in \cite{Rscripts_GitHub:25}.  This algorithm uses the characterization of Lemma \ref{lemma:char_MLE}. 

The asymptotic distribution of the nonparametric MLE in the IC-$m$ model has only partially been established for the case 2 model in  the so-called separated case,  in which the distance between observation times stays away from zero. However, its proof is given under the additional assumption that the distance between the observation times has an absolutely  continuous distribution, a situation we clearly do not have here.  For details, see \cite{piet:96}.
One conjecture for the asymptotic distribution of the MLE in the uniform deconvolution problem when $F_0(1)<1$ would be that the statement of Theorem \ref{th:local_limit} is also valid in this setting.
In the next section, we consider the more general mixed uniform deconvolution problem for which asymptotic results have been derived leading to a second conjecture for the asymptotic distribution in the general fixed uniform deconvolution problem.

\section{The mixed model}
\label{sec:mixedmodel}
The {\it mixed uniform deconvolution model} as described in the introduction (so with the random length $E$ of the support of the uniform noise $V$) is clearly a generalization of the what we from now will call the {\it fixed uniform deconvolution model}. The first reducing to the latter in case  $E$ has the degenerate distribution at the point $1$. Let  $M_1$ and $M_2$ be the upper support points of the distributions corresponding to $E$ and $U$ respectively, so that $M=M_1+M_2$ is the upper support point of the distribution of $S=U+V$. Denoting by $\m$ the product of the measure $dF_E$ and Lebesgue measure on $[0,M]$, $(E_i,S_i)$
has (convolution) density $q_F$ with respect to $\m$
\begin{align}
\label{convolution}
q_F(e,s)&=e^{-1}\{F(s)-F(s-e)\}\nonumber\\
&=e^{-1}\int_{u=(s-e)_+}^s \,dF(u),\qquad e>0,\,s\in[0,M].
\end{align}
We define the underlying measure $Q_0$ for $(E_i,S_i)$ by
\begin{align}
\label{def_Q_0}
dQ_0(e,s)=q_F(e,s)\,ds\,\,dF_E(e),\qquad s\in[0,M],\qquad e\in(0,M_2].
\end{align}

For estimating the distribution function $F_0$ of $U$, usually parametric distributions are used, like the Weibull, log-normal or gamma distribution. However, in \cite{piet:21} the nonparametric MLE is used. This maximum likelihood estimator $\hat F_n$ maximizes the function
\begin{align}
\label{loglikelihood}
\ell(F)=n^{-1}\sum_{i=1}^n\log\left(F(S_i)-F(S_i-E_i)\right)
\end{align}
over {\it all} distribution functions $F$ on $\R$ which satisfy $F(x)=0$, $x\le 0$, see \cite{piet:21}. The monotonicity and boundedness of $F$ (between $0$ and $1$) ensures that this maximization problem has a solution.
Note that the random variable $E_i$ now replaces the value $1$ in $F(S_i-1)$ compared to the fixed uniform deconvolution problem, but that for the rest the model is the same.

The asymptotic distribution of the nonparametric MLE in the mixed uniform deconvolution problem is given as Theorem 4.1 in \cite{piet_EJS:24}.
\begin{theorem}
\label{th:asympmixed} 
Let the conditions of Theorem 4.1 in \cite{piet_EJS:24} be satisfied. In particular, let $E_i$ stay away from zero and have an absolutely continuous distribution. Then
\begin{align}
\label{local_limit_resultmix}
n^{1/3}\{\hat F_n(t_0)-F_0(t_0)\}/(4f_0(t_0)/c_E)^{1/3}\stackrel{d}\longrightarrow \text{\rm argmin}\left\{W(t)+t^2\right\},
\end{align}
where $W$ is two-sided Brownian motion on $\R$, originating from zero, and where the constant $c_E$ is given by:
\begin{align}
\label{c_E}
c_E=\int e^{-1}\left[\frac1{F_0(t_0)-F_0(t_0-e)}
+\frac1{F_0(t_0+e)-F_0(t_0)}\right]\,dF_E(e).
\end{align}
\end{theorem}
This result is valid, irrespective of whether the support of distribution corresponding to $F_0$ is contained in $[0,1]$ or not.

Here we see a condition for that result, namely that $E$ has an absolutely continuous distribution, a condition clearly not met by the degenerate distribution at $1$. Ignoring this for the moment, the statement of the theorem would yield the asymptotic distribution with variance 
\begin{align*}
\left(4f_0(t_0)/c_E\right)^{2/3}\text{Var}\left(\text{\rm argmin}_{t\in\R}\left\{W(t)+t^2\right\}\right),
\end{align*}
where,
\begin{align}
c_E&=\int e^{-1}\left[\frac1{F_0(t_0)-F_0(t_0-e)}
+\frac1{F_0(t_0+e)-F_0(t_0)}\right]\,dF_E(e)\nonumber\\
\label{eq:cE1}
&=\frac1{F_0(t_0)-F_0(t_0-1)}+\frac1{F_0(t_0+1)-F_0(t_0)}.
\end{align}
Now, if $F_0(1)=1$, this expression (for $t_0\in(0,1)$) reduces to
\begin{align}
\label{eq:cE2}
c_E=\frac1{F_0(t_0)}+\frac1{1-F_0(t_0)}=\frac1{F_0(t_0)(1-F_0(t))}.
\end{align}
This is precisely the constant in the asymptotic variance of Theorem \ref{th:local_limit}, proved in the case $F_0(1)=1$. Now two natural conjectures come up. The first is that also for the setting where $F_0(1)<1$  for $c_E$ will equal (\ref{eq:cE2}). The second being that it equals (\ref{eq:cE1}).  In Section \ref{sec:asymptandsimul} we will further investigate this asymptotic variance, using a simulation study as well as  asymptotic calculations involving so-called smooth functionals. In section \ref{sec:examples}, we discuss asymptotic results for this type of functionals.

\section{Interlude; examples of applications of smooth functional theory in the fixed model}
\label{sec:examples}
Though estimating the distribution function at a fixed point in the uniform deconvolution function is intrinsically more complicated than based on a direct sample from the unknown distribution, reflected in a slower rate of estimation than the `parametric rate' $\sqrt{n}$, there are functionals of $F_0$ that can be estimated at rate $\sqrt{n}$ in the uniform deconvolution model.
For the functionals discussed in this section, we have known (normal) asymptotic behavior of the functions of the nonparametric MLE. The theory does not depend on whether the support of the distribution, corresponding to $F_0$, is contained in $[0,1]$ or not.  
We illustrate the theory  with $F_0$ the truncated exponential distribution function on $[0,2]$, given by
\begin{align}
\label{truncated_exp}
F_0(x)=\frac{1-\exp\{-x\}}{1-\exp\{-2\}}1_{[0,2]}(x)+1_{(2,\infty)}(x).
\end{align}
Note that in this case the support of $f_0$ is not contained in $[0,1]$.

By ``smooth functionals" we mean both global smooth functionals such as the ``mean functional" and local smooth functionals such as smooth approximations of the distribution function and density, which can be estimated by applying kernel smoothing to the nonparametric MLE. These estimates are asymptotically normal. For deriving the asymptotic behavior, we establish a representation in the ``observation space'' of the functional in the ``hidden space'', using the score function. The setup is briefly described in the appendix (section \ref{subsection:score_equations}). For a more detailed discussion of the smooth functional theory, see \cite{piet_geurt:14}.

\begin{example}
\label{example:mean}
{\rm
Using the theory of Section \ref{section:Appendix} for the functional $F\mapsto\int x\,dF(x)$, we get from Lemma \ref{lemma:solution_score_eq} for the score function $\th_F$ on $[0,3]$:
\begin{align}
\label{score_function_F}
\th_F(x)=\left\{\begin{array}{ll}
-\{1-F(x)\}-\{1-F(x+1)\},\qquad &x\in[0,1],\\
F(x-1)-\{1-F(x)\}, &x\in(1,2],\\
F(x-2)+F(x-1), &x\in(2,3].\\
\end{array}
\right.
\end{align}
For $F=F_0$ as in (\ref{truncated_exp}), this becomes:
\begin{align}
\label{score_function_truncated_exp}
\th_F(x)=\left\{\begin{array}{ll}
\bigl(2 e^x-e (1+e)\bigr)/\bigl((e^2-1) e^x\bigr),\qquad\qquad\qquad    &x\in[0,1],\\
\bigl((1+e^2) e^x-e^2 (1+e)\bigr)/\bigl((e^2-1) e^x\bigr), &x\in(1,2],\\
(2 e^2-(1+e) e^{3-x})/(e^2-1), &x\in(2,3].\\
\end{array}
\right.
\end{align}
A picture of this function is given in Figure \ref{figure:truncated_exp_score}.

\begin{figure}[!ht]
\centering
\includegraphics[width=0.5\textwidth]{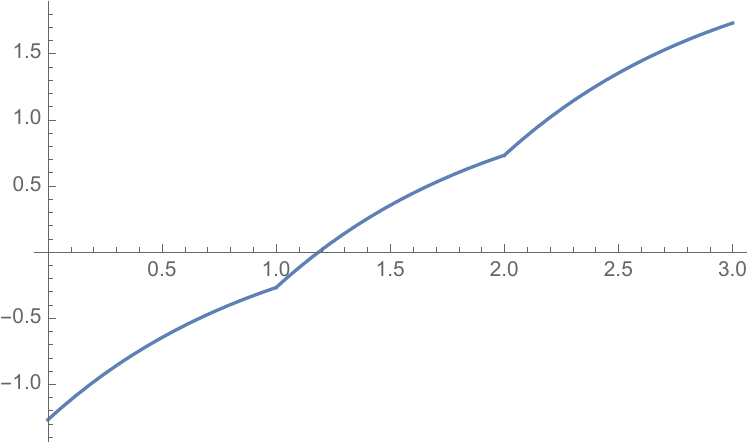}
\caption{The score function $\th_{F_0}$ for the truncated exponential distribution, given by (\ref{score_function_truncated_exp}).}
\label{figure:truncated_exp_score}
\end{figure}
The asymptotic variance of $\sqrt{n}\int x\,d\hat F_n(x)$ is given by:
\begin{align*}
\s_{F_0}^2=\int_0^3\th_{F_0}(x)^2\{F_0(x)-F_0(x-1)\}\,dx\approx0.357915,
\end{align*}
and we have:
\begin{align*}
\sqrt{n}\left\{\int x\,d\hat F_n(x)-\int x\,dF_0(x)\right\}\stackrel{{\cal D}}\longrightarrow N(0,\s_{F_0}^2),
\end{align*}
where $N(0,\s_{F_0}^2)$ is a normal distribution with expectation zero and variance $\s_{F_0}^2$.
One can in fact consistently estimate the asymptotic variance $\s_{F_0}^2$ by
\begin{align*}
\int\th_{\hat F_n}(x)^2\{\hat F_n(x)-\hat F_n(x-1)\}\,dx,
\end{align*}
where one takes $F=\hat F_n$ in (\ref{score_function_F}). Note that the rate of convergence is $\sqrt{n}$ instead of the cube root $n$ rate, expected for $\hat F_n(t)$ itself.

The asymptotic efficiency of this estimate of the first moment is proved in  \cite{geer:00}, see in particular example 11.2.3e on p.\ 230 of \cite{geer:00}. It beats obvious moment estimates like $\bar S_n-1/2$, because of the particular form of the characteristic function of the Uniform distribution (having zeroes). But these matters are not the main concern of our paper.
}
\end{example}

\begin{example}
\label{example:density}
{\rm
We can estimate the density $f_0$ by the density estimator
\begin{align*}
\hat f_{nh}(t)=\int K_h(t-x)\,d\hat F_n(x)
\end{align*}
where $h>0$ is a bandwidth and
$K_h(u)=h^{-1}K(u/h)$ for a symmetric smooth kernel $K$, for example the triweight kernel
\begin{align*}
K(u)=\tfrac{35}{32}(1-u^2)^31_{[-1,1]}(u).
\end{align*}
This estimator was considered in \cite{GroJoUnifDec:03}. It was shown that, under the condition given there:
\begin{align*}
\bigl(nh^3\bigr)^{1/2}\left\{\hat f_{nh}(t)-\int K_h(t-x)\,dF_0(x)\right\}\stackrel{{\cal D}}\longrightarrow N(0,\s_t^2),
\end{align*}
where
\begin{align}
\label{sigma_squared}
\s_t^2=\lim_{h\downarrow0}h^3\int\th_{h,t,F_0}^2(s)\{F_0(s)-F_0(s-1)\}\,ds,
\end{align}
and $\th_{h,t,F_0}$ is given by
\begin{align*}
\th_{h,t,F_0}(x)=\sum_{i=0}^{m-1}\{1-F_0(x+i)\}K_h'(t-(x+i)),\qquad x\in[0,1].
\end{align*}
Here $\th_{h,t,F_0}(x+i)=\th_{h,t,F_0}(x+i-1)-K_h'(t-(x+i-1))$, $i=1\dots,m$ and $m=\lceil M\rceil$, where $M$ is the upper bound of the support of $f_0$.

We show that $\s_t^2$, given by (\ref{sigma_squared}) can be simplified to
\begin{align}
\label{sigma_squared2}
\s_t^2=F_0(t)\{1-F_0(t)\}\int K'(u)^2\,du.
\end{align}
We have, for $t\in(0,1)$:
\begin{align*}
&\int\th_{h,t,F_0}^2(x)\{F_0(x)-F_0(x-1)\}\,dx\\
&=\int_{x=t-h}^{t+h} K_h'(t-x)^2\{1-F_0(x)\}^2F_0(x)\,dx\\
&\qquad+\int_{x=t-h}^{t+h} K_h'(t-x)^2F_0(x)^2\{F_0(x+1)-F_0(x)\}\,dx\\
&\qquad+\dots\\
&\qquad+\int_{x=t-h}^{t+h} K_h'(t-x)^2F_0(x)^2\{1-F_0(x+m-1)\}\,dx\\
&=\int_{x=t-h}^{t+h} K_h'(t-x)^2\{1-F_0(x)\}^2F_0(x)\,dx\\
&\qquad\qquad+\int_{x=t-h}^{t+h} K_h'(t-x)^2F_0(x)^2\{1-F_0(x)\}\,dx\\
&=\int_{x=t-h}^{t+h} K_h'(t-x)^2F_0(x)\{1-F_0(x)\}\,dx\\
&\sim h^{-3}F_0(t)\{1-F_0(t)\}\int K'(u)^2\,du,\qquad h\downarrow0,
\end{align*}
using the telescoping sum $F_0(x+1)-F_0(x)+\dots+1-F_0(x+m-1)$.

For $t\in(i,i+1)$, $i\ge1$ we get a similar argument and for $t=i$, $i=1,2,\dots$ we get the result from the contributions from both sides of $i$, which each contribute one half to the total mass.
}
\end{example}

\begin{example}
\label{example:df}
{\rm
In an entirely similar way, the asymptotic distribution of the estimator
\begin{align*}
\tilde F_{nh}(t)=\int\IK_h(t-x)\,d\hat F_n(x)
\end{align*}
of the distribution function $F_0$ itself can be derived, where $\IK_h(u)=\IK(u/h)$ and $\IK$ is the integrated kernel $K$, defined by
\begin{align*}
\IK(u)=\int_{-\infty}^u K(x)\,dx.
\end{align*}
We have the following result.

\begin{theorem}
\label{theorem:limit_df}
Let $F_0$ have a continuous positive density $f_0$ on $(0,M)$, where $M>0$, and let  $t_0\in(0,M)$. Then, if $nh\to\infty$ and $h\downarrow0$:
\begin{align*}
\bigl(nh\bigr)^{1/2}\left\{\tilde F_{nh}(t)-\int \IK_h(t-x)\,dF_0(x)\right\}\stackrel{{\cal D}}\longrightarrow N(0,\tilde\s_t^2),
\end{align*}
where
\begin{align*}
\tilde\s_t^2 =F_0(t)\{1-F_0(t)\}\int K(u)^2\,du.
\end{align*}
\end{theorem}
The proof proceeds along a similar path as the proof of the analogous result for the density in the preceding example. The truncated exponential distribution function, defined by (\ref{truncated_exp}) satisfies the conditions of the theorem, for $M=2$.

}
\end{example}

\section{Simulations and asymptotic considerations}
\label{sec:asymptandsimul}
Having two competing conjectures for the asymptotic distribution of $\hat{F}_n(t_0)$ in the fixed uniform deconvolution problem, related to the constants (\ref{eq:cE1})  and (\ref{eq:cE2}) in the asymptotic variance, in this section we will first present a simulation study supporting  (\ref{eq:cE2}). One could say that ``plugging in'' the point mass at $1$ as distribution for $E$ should be possible in the mixed case setting in case $F_0(1)=1$. In the case $F_0(1)<1$, however, this does not seem to lead to the right result. Diving into the proof of Theorem 4.1 in \cite{piet_EJS:24}, we identify a term that in case $F_0(1)=1$ is identically zero. But if $F_0(1)<1$ this term is not zero and has order  $O_p(n^{-5/6})$ in the mixed model, but order $O_p(n^{-2/3})$ in the fixed model, which means that it is negligible in the mixed model but not negligible in the asymptotic expansion for the fixed model.

To simulate the mixed model, we first generate the random pairs
\begin{align*}
(U_i,E_i),\qquad U_i\sim F_0,\qquad E_i\sim F_E=\left(\cdot-0.5\right)1_{[0.5,1.5)}+1_{[1.5,\infty)},\qquad i=1,\dots,n.
\end{align*}
Then, conditionally on $E_i$ a Uniform$(0,E_i)$ random variable $V_i$ is drawn. Our observations then consist of
\begin{align*}
(E_i,S_i),\,\,\mbox{ where }\,\,S_i=U_i+V_i,\qquad i=1,\dots,n.
\end{align*}
This example satisfies the conditions of Theorem 4.1 in \cite{piet_EJS:24}, which means that the asymptotic variance of the nonparametric maximum likelihood estimator $\hat F_n$ of $F_0$ evaluated at $t_0\in(0,2)$ is given by:
\begin{align}
\label{sigma_t}
\s_{t_0}^2=\left(4f_0(t_0)/c_E\right)^{-2/3}\text{var}\left(\text{\rm argmin}_{t\in\R}\left\{W(t)+t^2\right\}\right),
\end{align}
where $c_E$ is given by (\ref{c_E}).
It was computed numerically in \cite{piet_jon:01}, using Airy functions, that
\begin{align*}
\text{var}\left(\text{\rm argmin}_{t\in\R}\left\{W(t)+t^2\right\}\right)\approx 0.263555964.
\end{align*}

To illustrate Theorem 4.1 in \cite{piet_EJS:24} we can simulate from the model and compute the variances of $\hat F_n(t_i)$ for, say $10,000$, samples for $t_i=i\cdot 0.1\,$, $i=1,\dots,19$ and compare the variances of $\hat F_n(t_i)$ with the asymptotic variances $\s^2_{t_i}$, taking $t_0=t_i$ in (\ref{sigma_t}). A comparison of simulated and asymptotic values is shown in Figure \ref{figure:variances_incub}, where we take $n=10,000$, and  $F_0$ the truncated standard exponential and the uniform distribution function on $[0,2]$, respectively.

\begin{figure}[!ht]
 		\begin{subfigure}[b]{0.4\textwidth}
 			\includegraphics[width=\textwidth]{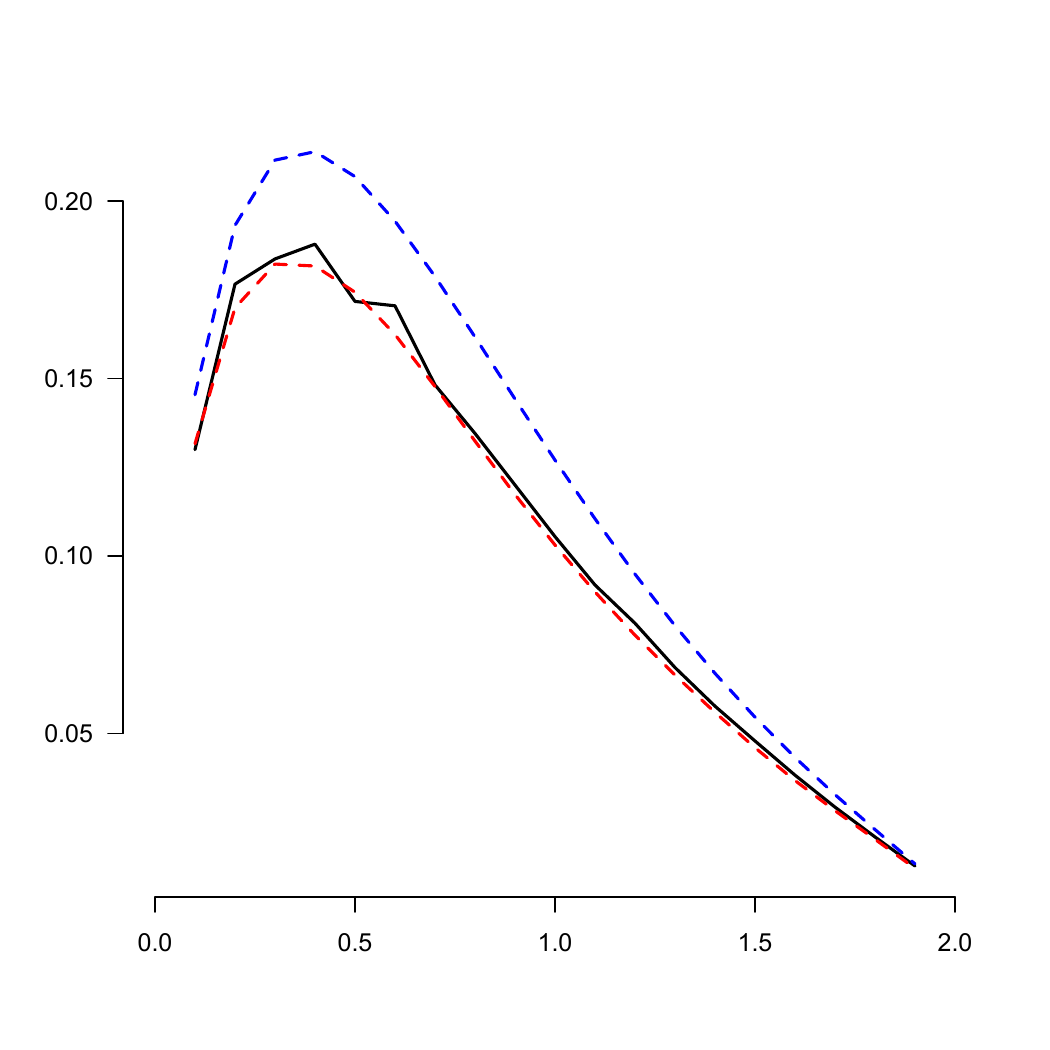}
 			\caption{}
 			\label{fig:posterior_estimate}
 		\end{subfigure}
 		\begin{subfigure}[b]{0.4\textwidth}
 			\includegraphics[width=\textwidth]{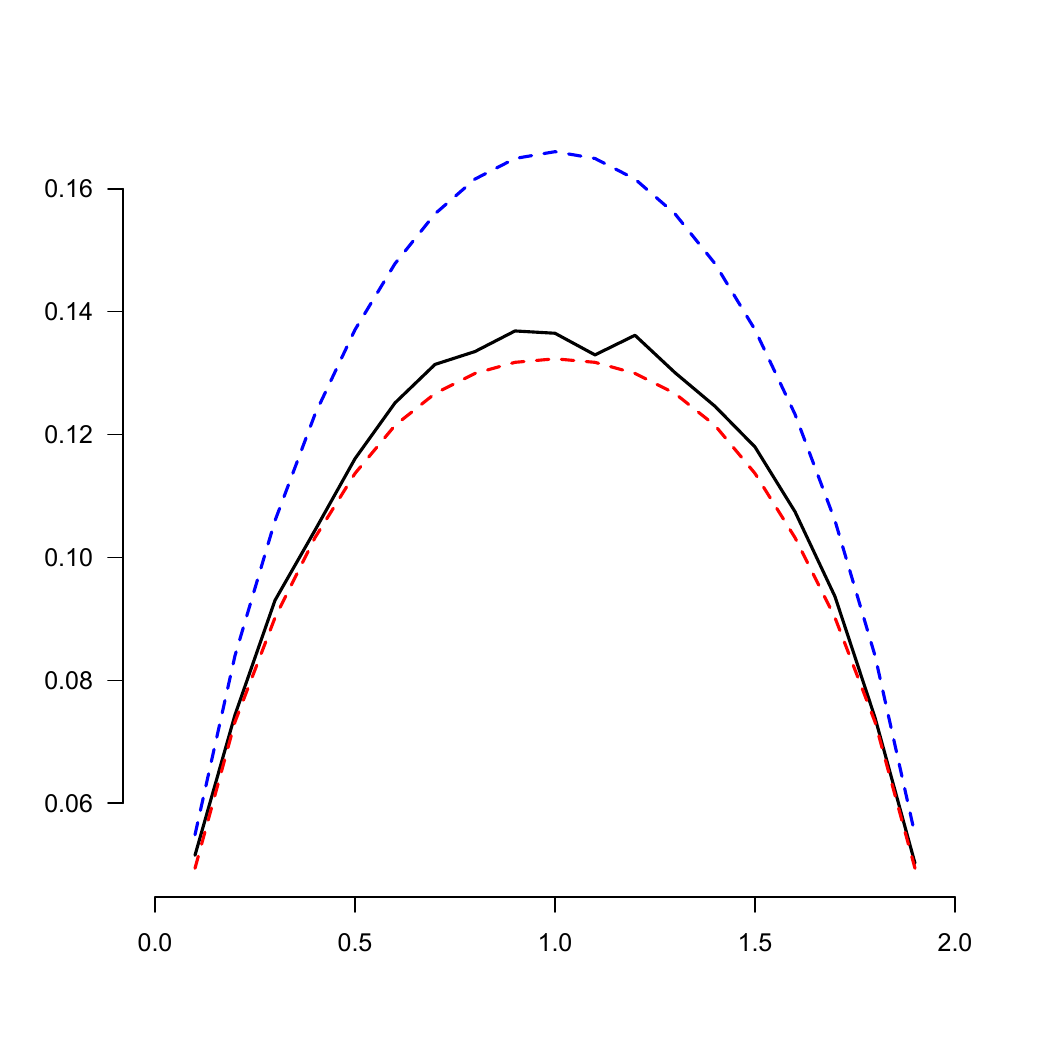}
 			\caption{}
 			\label{fig:isotonic_projection_posterior}
 		\end{subfigure}
 		\caption{Simulated variances, times $n^{2/3}$, for the mixed model of $\hat F_n(t_i)$, for $t_i=0.1,0.2,\dots,1.9$ (blue solid curve, linearly interpolated between values at the $t_i$), compared with the asymptotic values (\ref{sigma_t}) of Theorem 4.1 in \cite{piet_EJS:24} (red, dashed) for $E_i$ uniform on $[0.5,1.5]$. The simulated variances are based on $10,000$ simulations of samples of size $n=10,000$ for (a) $F_0$ the truncated exponential distribution function on $[0,2]$ and (b) $F_0$ the uniform distribution function on $[0,2]$. The blue dashed curves are the corresponding asymptotic variance curves of Conjecture \ref{conjecture:limit_df} for the fixed model.}
\label{figure:variances_incub}
 	\end{figure}

If $M>1$, as in this example where $M=2$, the values of the asymptotic variances are no longer of the form (\ref{sigma_t}) in the fixed model. We shall now explain the reason of this discrepancy.
We start with the characterization of Lemma \ref{lemma:char_MLE}. The MLE is determined by the process $W_{\hat F_n}$ and, just as in \cite{piet_EJS:24}, we now investigate the terms in a local expansion of $W_{\hat F_n}$. To this end, we first define the process $X_n$.

\begin{lemma}
\label{lemma:BM_part}
Let $F_0$ have a continuous positive density $f_0$ on $(0,M)$, where $M>1$, let  $t_0\in(0,M)$ and let the process $X_n$ be defined by:
\begin{align}
\label{def_X_n}
X_n(t)&=\int_{t_0<s\le t_0+n^{-1/3}t}\frac{\d_1(s)}{\hat F_n(s)}\,d(\Q_n-Q_0)(s)\nonumber\\
&\qquad\qquad-\int_{t_0<s-1\le t_0+n^{-1/3}t}\frac{\d_2(s)}{\hat F_n(s)-\hat F_n(s-1)}\,d(\Q_n-Q_0)(s)\nonumber\\
&\qquad\qquad+\int_{t_0<s\le t_0+n^{-1/3}t}\frac{\d_2(s)}{\hat F_n(s)-\hat F_n(s-1)}\,d(\Q_n-Q_0)(s)\nonumber\\
&\qquad\qquad-\int_{t_0<s-1\le t_0+n^{-1/3}t}\frac{\d_3(s)}{1-\hat F_n(s-1)}\,d(\Q_n-Q_0)(s),
\end{align}
where $\d_1=1_{[0,1]}$, $\d_2=1_{(1,m_n]}$ and $\d_3=1_{(m_n,\infty)}$ and $\Q_n$ is the empirical distribution measure of the $S_i$ with corresponding underlying measure $Q_0$. 
Then $n^{2/3}X_n$ converges in distribution, in the Skorohod topology, to the process
\begin{align*}
t\mapsto \sqrt{c}W(t),\qquad t\in\R,
\end{align*}
where $c$ is defined by
\begin{align*}
c=\frac1{F_0(t_0)-F_0(t_0-1)}+\frac1{F_0(t_0+1)-F_0(t_0)}\,.
\end{align*}
\end{lemma}
\begin{proof}
The proof follows the proof of Lemma 9.4 in \cite{piet_EJS:24}.
\end{proof}

We now have:
\begin{lemma}
\label{lemma:W_n-expansion}
\begin{align*}
W_{n,\hat F_n}(t_0+n^{-1/3}t)-W_{n,\hat F_n}(t_0)=X_n(t)+Y_n(t),
\end{align*}
where $X_n$ is defined by (\ref{def_X_n}) and $Y_n$ by
\begin{align}
\label{def_Y_n2}
Y_n(t)=\int_{t_0<s\le t_0+n^{-1/3}t}\left\{\frac{F_0(s)-F_0(s-1)}{\hat F_n(s)-\hat F_n(s-1)}
-\frac{F_0(s+1)-F_0(s)}{\hat F_n(s+1)-\hat F_n(s)}\right\}\,ds.
\end{align}
We can write
\begin{align*}
Y_n(t)=A_n(t)+B_n(t),
\end{align*}
where
\begin{align*}
&A_n(t)\\
&=-\int_{s\in[t_0,t_0+n^{-1/3}t)}\bigl\{\hat F_n(s)-F_0(s)\bigr\}\left\{\frac{1}{\hat F_n(s)-\hat F_n(s-1)}+\frac{1}{\hat F_n(s+1)-\hat F_n(s)}\right\}\,ds,
\end{align*}
and
\begin{align}
\label{B_n(t)}
B_n(t)=\int_{s\in[t_0,t_0+n^{-1/3}t)}\left\{\frac{\hat F_n(s-1)-F_0(s-1)}{\hat F_n(s)-\hat F_n(s-1)}
+\frac{\hat F_n(s+1)-F_0(s+1)}{\hat F_n(s+1)-\hat F_n(s)}\right\}\,ds.
\end{align}
\end{lemma}

The crucial difference between the mixed model and the fixed model is the term $B_n(t)$ in the expansion of $W_{n,\hat F_n}$ in Lemma \ref{lemma:W_n-expansion}. In the mixed model the corresponding term
\begin{align}
\label{B_n(t)_mixed}
&\tilde B_n(t)\nonumber\\
&=\int_{s\in[t_0,t_0+n^{-1/3}t)}\int e^{-1}\left\{\frac{\hat F_n(s-e)-F_0(s-e)}{\hat F_n(s)-\hat F_n(s-e)}
+\frac{\hat F_n(s+e)-F_0(s+e)}{\hat F_n(s+e)-\hat F_n(s)}\right\}\,dF_E(e)\,ds
\end{align}
is shown in \cite{piet_EJS:24} to be of order $O_p(n^{-5/6})$ in the leading expansion, but (\ref{B_n(t)}) is of order $O_p(n^{-2/3})$.
In fact, the terms $A_n(t)$ and $B_n(t)$ are both of order $O_p(n^{-2/3})$ in the fixed model, whereas the corresponding terms are of order $O_p(n^{-2/3})$ and $O_p(n^{-5/6})$, respectively, in the mixed model, see the proof of Lemma 9.8 and Lemma 9.9 in \cite{piet_EJS:24}, respectively. Integration w.r.t.\ $dF_E$, where $F_E$ is absolutely continuous, causes the lower order of the term. Showing that (\ref{B_n(t)_mixed}) is of order $O_p(n^{-5/6})$, however, is not at all easy and is in fact the main effort in \cite{piet_EJS:24}. We need to apply the smooth functional theory here, as discussed in Section \ref{sec:examples}.

We have the following conjecture on the asymptotic behavior of $\hat F_n$ if  $V_i$ is Uniform$(0,1)$.

\begin{conjecture}
\label{conjecture:limit_df}
Let $F_0$ have a continuous positive density $f_0$ on $(0,M)$, where $M>0$, and let  $t_0\in(0,M)$. Let $\hat F_n$ be the nonparametric MLE of $F_0$. Then, for $t_0\in(a,b)$:
\begin{align}
\label{local_limit_resultConject}
n^{1/3}\{\hat F_n(t_0)-F_0(t_0)\}/\left(4f_0(t_0)F_0(t_0)(1-F_0(t_0))\}\right)^{1/3}\stackrel{d}\longrightarrow \text{\rm argmin}_{t\in\R}\left\{W(t)+t^2\right\},
\end{align}
where $W$ is two-sided Brownian motion on $\R$, originating from zero.
\end{conjecture}

\begin{figure}[!ht]
 		\begin{subfigure}[b]{0.4\textwidth}
 			\includegraphics[width=\textwidth]{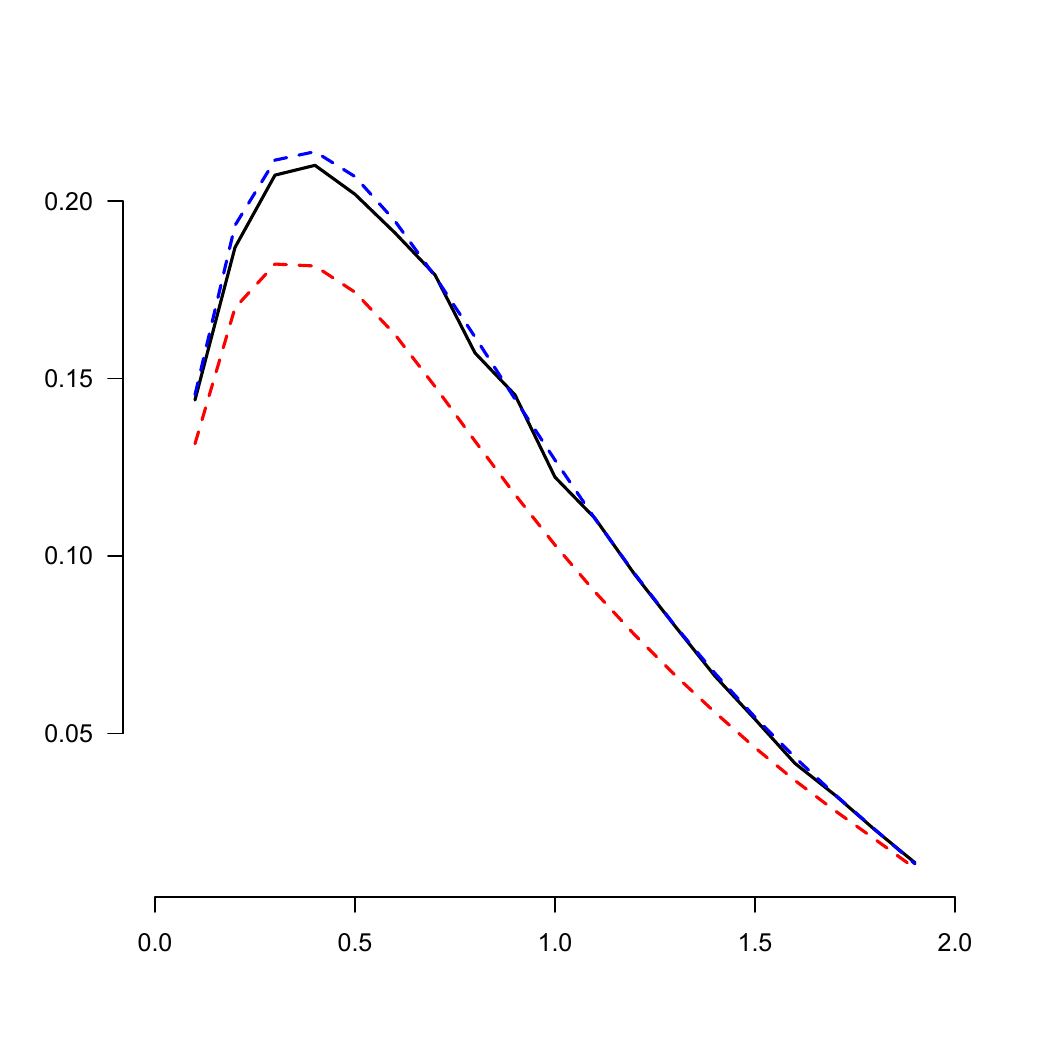}
 			\caption{}
 			\label{fig:fixed_10,000exp}
 		\end{subfigure}
 		\begin{subfigure}[b]{0.4\textwidth}
 			\includegraphics[width=\textwidth]{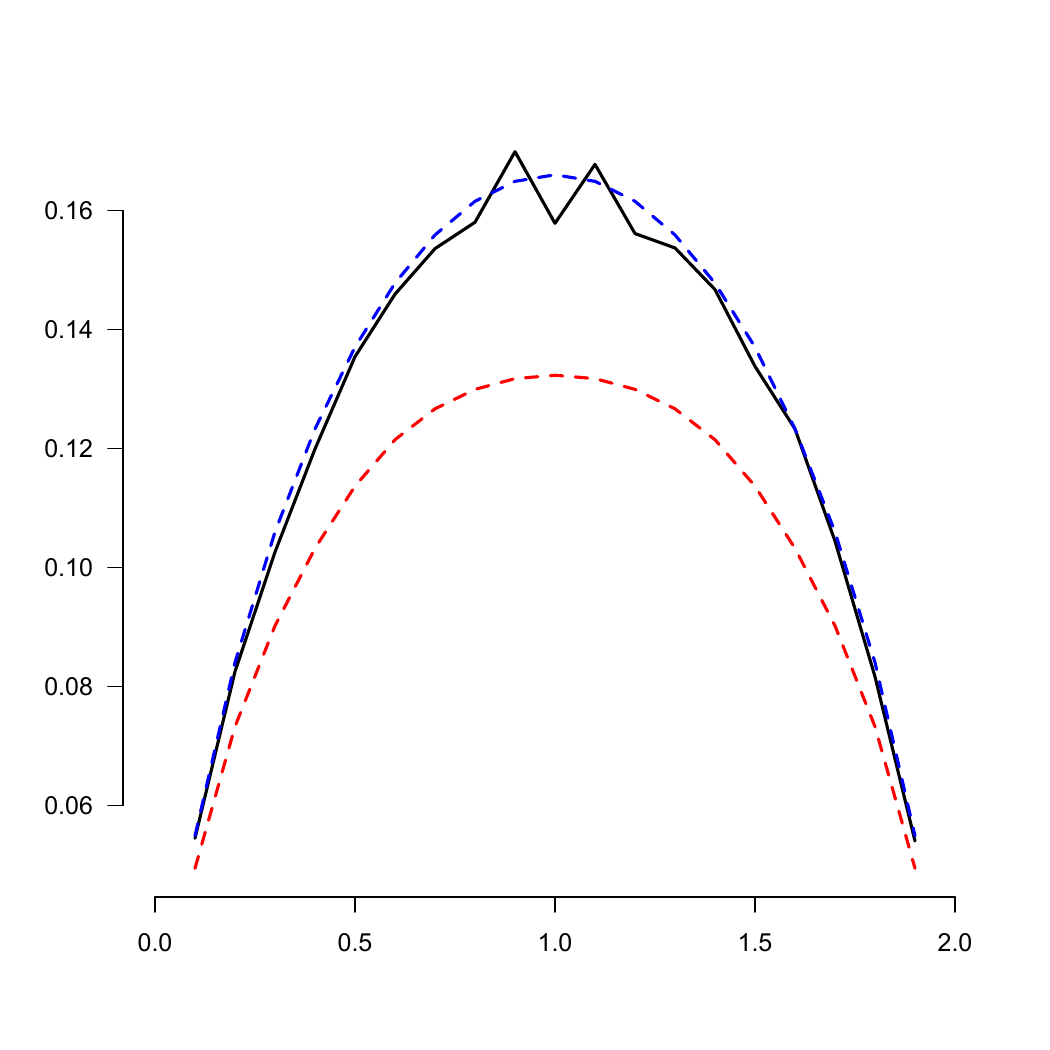}
 			\caption{}
 			\label{fig:fixed_10,000_uniform}
 		\end{subfigure}
 		\caption{Simulated variances, times $n^{2/3}$, for the fixed model of $\hat F_n(t_i)$, for $t_i=0.1,0.2,\dots,1.9$ (blue solid curve, linearly interpolated between values at the $t_i$), compared with the asymptotic values of Conjecture \ref{conjecture:limit_df} (blue, dashed). The simulated variances are based on $10,000$ simulations of samples of size $n=10,000$ for (a) $F_0$ the truncated exponential distribution function on $[0,2]$ and (b) $F_0$ the uniform distribution function on $[0,2]$. The red dashed curves are the corresponding asymptotic variance curves of Theorem 4.1 in \cite{piet_EJS:24} for the mixed model and $E_i$ uniform on $[0.5,1.5]$.}
 	\end{figure}

So in this case we expect the MLE to have exactly the same limit behavior as in the case that the support of the distribution is contained in $[0,1]$ (see Theorem \ref{th:local_limit}). The reason we believe the conjecture might be true partly relies on simulations and partly on the behavior of the smooth functionals in Section \ref{sec:examples}, where also the $F(1-F)$ ``Bernoulli factor" enters into the variance. If the conjecture is true, we think it enters via telescoping sums, as we demonstrated for the density estimator, based on the MLE, but we have not been able to prove this.

We finally show in Figure \ref{figure:2conjectures} a picture of the variance curve for 10,000 samples of size $n=10,000$, where the blue dashed curve shows the theoretical curve of Conjecture \ref{conjecture:limit_df} and the purple dotted curve the theoretical curve if we would apply Theorem 4.1 in \cite{piet_EJS:24} with $E$ degenerate at $1$ (ignoring the conditions of Theorem 4.1 in \cite{piet_EJS:24}).

\begin{figure}[!ht]
 		\begin{subfigure}[b]{0.4\textwidth}
 			\includegraphics[width=\textwidth]{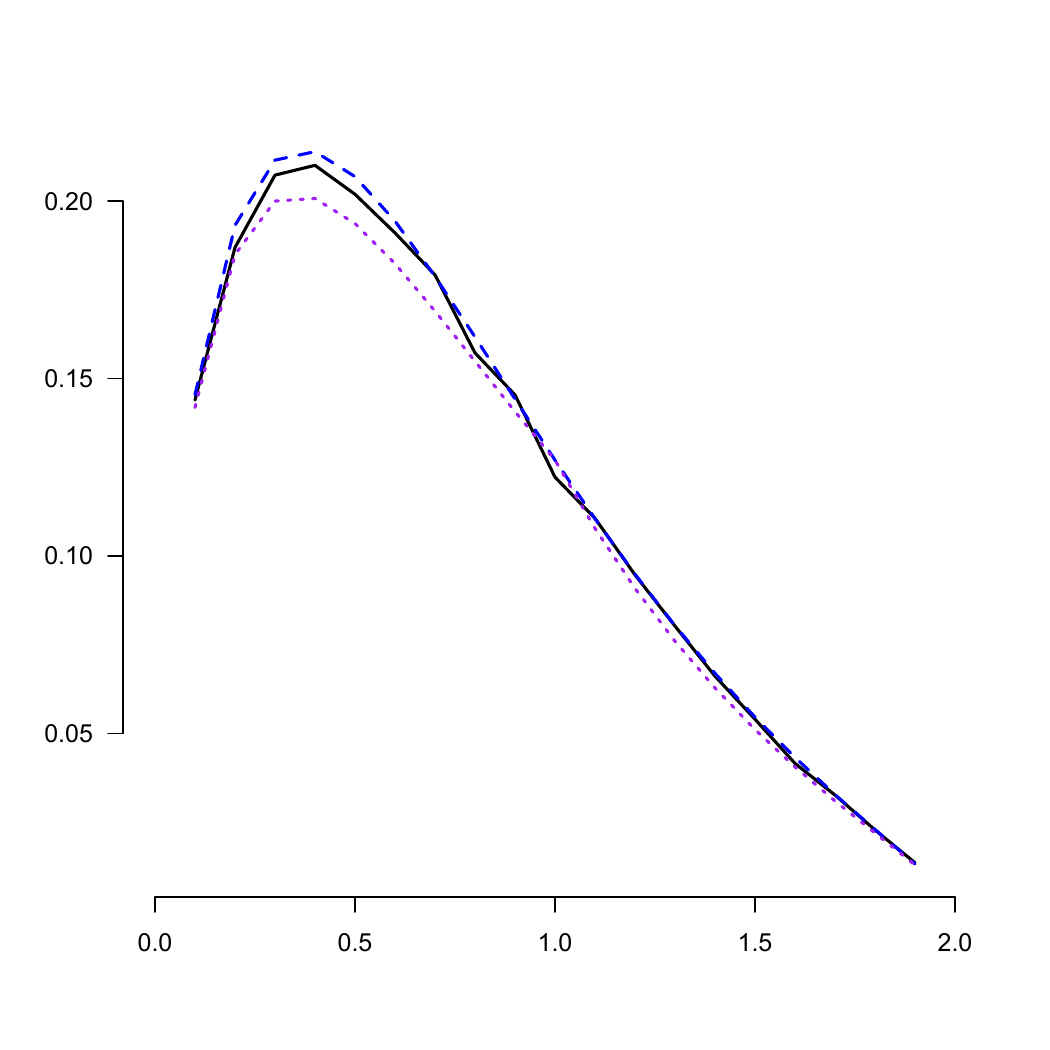}
 			\caption{}
 		\end{subfigure}
 		\begin{subfigure}[b]{0.4\textwidth}
 			\includegraphics[width=\textwidth]{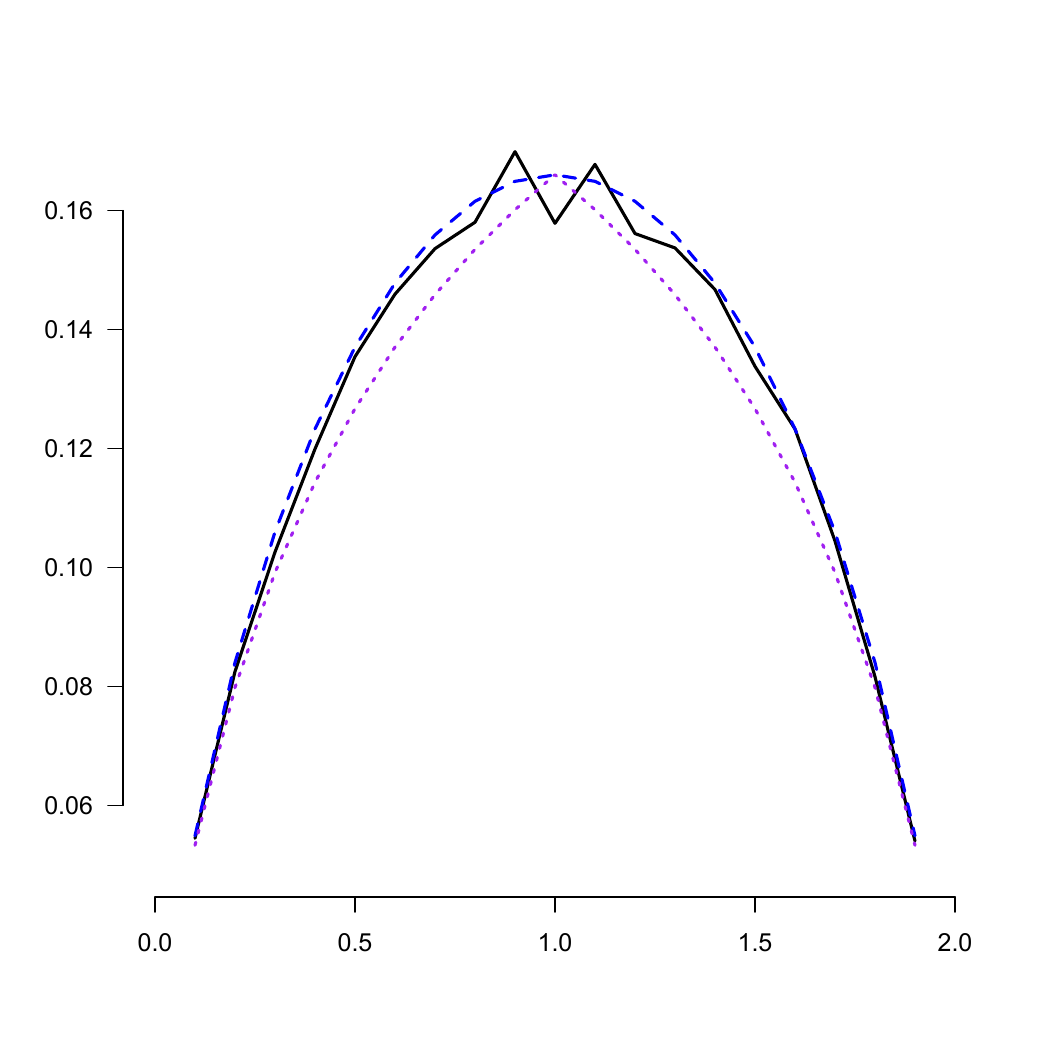}
 			\caption{}
 		\end{subfigure}
 		\caption{Simulated variances, times $n^{2/3}$, for the fixed model of $\hat F_n(t_i)$, for $t_i=0.1,0.2,\dots,1.9$ (black solid curve, linearly interpolated between values at the $t_i$), compared with the asymptotic values of Conjecture \ref{conjecture:limit_df} (blue, dashed) and the conjecture that would follow from Theorem 4.1 in \cite{piet_EJS:24}, ignoring the conditions (purple, dotted), for (a) $F_0$ the truncated exponential distribution function on $[0,2]$ and (b) $F_0$ the uniform distribution function on $[0,2]$.
 The simulated variances are based on $10,000$ simulations of samples of size $n=10,000$.}
 \label{figure:2conjectures}
 	\end{figure}
	
	The ``dip'' of the empirical variance curves at the point $1$ if $F_0$ is the uniform distribution function on $[0,2]$, which was also visible for the mixed model in Figure \ref{figure:variances_incub}, is remarkable, but we do not have an explanation for it. Note that the theoretical curves touch at this point, because the theoretical variance at the point $t=1$ contains $F_0(t)\{1-F_0(t\}$ for both conjectures there.

\section{Concluding remarks}
\label{section:conclusion}
We showed that the asymptotic distribution of the nonparametric maximum likelihood estimator (MLE) in the uniform deconvolution model can de derived from a corresponding result for the current status model in the case that the support of the distribution is contained in the unit interval. If the support of the unknown distribution is not contained in the unit interval, the asymptotic distribution is unknown. But also in this setting, the model is shown to be related to an interval censoring model, but now case $m$ for $m\ge 2$. 

The asymptotics in the mixed uniform deconvolution model (where the length of the support of the uniform variable is random) was studied in \cite{piet_EJS:24} under a smoothness condition on the distribution of the interval length. In that setting there is  no distinction depending on the support of the distribution corresponding to $F_0$. In case $F_0(1)=1$, the statement of the theorem reduces to that for the current status model if the (non-smooth) degenerate distribution is substituted. A natural question is whether the same can be done in case $F_0(1)<1$. 

Simulations indicate this cannot be done. Moreover, an important term in the asymptotic analysis of the nonparametric MLE in the mixed uniform deconvolution problem turns out to behave essentially differently, depending on whether $F_0(1)=1$ or $F_0(1)<1$. This discrepancy is explained from  smooth functional theory, the strength of which is also demonstrated using more easily understandable functionals in Section \ref{sec:examples}.

{\tt R} scripts for computing the MLE and producing the pictures in this paper are available in \cite{Rscripts_GitHub:25}.

\section{Appendix}
\label{section:Appendix}

\subsection{Score equations}
\label{subsection:score_equations}
As explained in the Appendix of \cite{piet:21}, the theory of the estimation of smooth functionals is based on certain score equations.
We define the score function $\th_F$ by:
\begin{align*}
\th_F(s)=E\left\{a(X)|S=s\right\}=\frac{\int_{x\in(s-1,s]} a(x)\,dF(x)}{F(s)-F(s-1)}\,, 
\end{align*}
(compare to (A.1) in \cite{piet:21}). This is the conditional expectation of $a(X)$ in the ``hidden'' space of the variable of interest, given our observation $S$. Note that we changed the notation somewhat w.r.t.\ \cite{piet:21}, and denote the distribution function of the incubation time by $F$ instead of $G$.

Defining
\begin{align*}
\f_F(t)=\int_{x\in[0,t]}a(x)\,dF(x),
\end{align*}
we get the following representation for the score function, conditioned on $X=x$:
\begin{align}
\label{original_representation}
E\{\th_F(S)|X=x\}
=\int_{s\in(x,x+1]}\frac{\f_F(s)-\f_F(s-1)}{F(s)-F(s-1)}\,ds.
\end{align}

\vspace{0.4cm}
By differentiation w.r.t.\ $x$ we get the following  equation, with on the right the derivative w.r.t.\ $x$ of the functional
 we want to estimate, denoted by $\psi$:
\begin{align}
\label{phi-eq}
\frac{\f_F(x+1)-\f_F(x)}{F(x+1)-F(x)}-\frac{\f_F(x)-\f_F(x-1)}{F(x)-F(x-1)}
=\psi(x),\qquad x\in[0,M].
\end{align}
Here and in the sequel, we assume that distribution functions are right-continuous.

We have the following lemma.

\begin{lemma}
\label{lemma:solution_score_eq}
The solution of the system (\ref{phi-eq}) is found in the following way. Let $m=\lceil M\rceil$ and let $\th_F(x)$ be given by
\begin{align*}
\th_F(x)=-\sum_{i=0}^{m-1}\{1-F(x+i)\}\psi(x+i),\qquad x\in[0,1].
\end{align*}
and $\th_F(x+i)=\th_F(x+i-1)+\psi(x+i-1)$, $i=1\dots,m+1$.

Then $\f_F$ is given by
\begin{align*}
\f_F(x)=F(x)\th_F(x),\qquad x\in[0,1],
\end{align*}
and
\begin{align*}
\f_F(x+i)-\f_F(x+i-1)=\{F(x+i)-\{F(x+i-1)\}\th_F(x+i),\qquad i=1,\dots,m.
\end{align*}
\end{lemma}

\begin{remark}
{\rm
This is in accordance with Example 11.2.3e on p.\ 230 of \cite{geer:00}, where $b_{F_0}=\th_{F_0}$. It is also in line with formula (12) in Theorem 2 in \cite{GroJoUnifDec:03}.
}
\end{remark}

\bibliographystyle{imsart-nameyear}
\bibliography{cupbook}

\begin{thebibliography}{15}

\bibitem[\protect\citeauthoryear{Backer, Klinkenberg and
  Wallinga}{2020}]{backer:20}
\begin{barticle}[author]
\bauthor{\bsnm{Backer},~\bfnm{Jantien~A.}\binits{J.~A.}},
  \bauthor{\bsnm{Klinkenberg},~\bfnm{Don}\binits{D.}} \AND
  \bauthor{\bsnm{Wallinga},~\bfnm{Jacco}\binits{J.}}
(\byear{2020}).
\btitle{{I}ncubation period of 2019 novel coronavirus (2019-n{C}ov) infections
  among travellers from {W}uhan, {C}hina, 20-28 January 2020}.
\bjournal{Euro Surveill.}
\bvolume{25}.
\end{barticle}
\endbibitem

\bibitem[\protect\citeauthoryear{Britton and
  Scalia~Tomba}{2019}]{tom_gianpi:19}
\begin{barticle}[author]
\bauthor{\bsnm{Britton},~\bfnm{Tom}\binits{T.}} \AND
  \bauthor{\bsnm{Scalia~Tomba},~\bfnm{Gianpaolo}\binits{G.}}
(\byear{2019}).
\btitle{Estimation in emerging epidemics: bases and remedies}.
\bjournal{J. R. Soc. Interface}
\bvolume{16}.
\end{barticle}
\endbibitem

\bibitem[\protect\citeauthoryear{Groeneboom}{1996}]{piet:96}
\begin{bincollection}[author]
\bauthor{\bsnm{Groeneboom},~\bfnm{P.}\binits{P.}}
(\byear{1996}).
\btitle{Lectures on inverse problems}.
In \bbooktitle{Lectures on probability theory and statistics ({S}aint-{F}lour,
  1994)}.
\bseries{Lecture Notes in Math.}
\bvolume{1648}
\bpages{67--164}.
\bpublisher{Springer}, \baddress{Berlin}.
\bdoi{10.1007/BFb0095675}
\bmrnumber{1600884 (99c:62092)}
\end{bincollection}
\endbibitem

\bibitem[\protect\citeauthoryear{Groeneboom}{2021}]{piet:21}
\begin{barticle}[author]
\bauthor{\bsnm{Groeneboom},~\bfnm{Piet}\binits{P.}}
(\byear{2021}).
\btitle{Estimation of the incubation time distribution for {COVID}-19}.
\bjournal{Stat. Neerl.}
\bvolume{75}
\bpages{161--179}.
\bdoi{10.1111/stan.12231}
\bmrnumber{4245907}
\end{barticle}
\endbibitem

\bibitem[\protect\citeauthoryear{Groeneboom}{2024a}]{piet_EJS:24}
\begin{barticle}[author]
\bauthor{\bsnm{Groeneboom},~\bfnm{Piet}\binits{P.}}
(\byear{2024}a).
\btitle{Nonparametric estimation of the incubation time distribution}.
\bjournal{Electron. J. Stat.}
\bvolume{18}
\bpages{1917--1969}.
\bdoi{10.1214/24-ejs2243}
\bmrnumber{4736274}
\end{barticle}
\endbibitem

\bibitem[\protect\citeauthoryear{Groeneboom}{2024b}]{piet_Statistica:24}
\begin{barticle}[author]
\bauthor{\bsnm{Groeneboom},~\bfnm{Piet}\binits{P.}}
(\byear{2024}b).
\btitle{Estimation of the incubation time distribution in the singly and doubly
  interval censored model}.
\bjournal{Stat. Neerl.}
\bvolume{78}
\bpages{617--635}.
\bdoi{10.1111/stan.12335}
\bmrnumber{4827421}
\end{barticle}
\endbibitem

\bibitem[\protect\citeauthoryear{Groeneboom}{2025}]{Rscripts_GitHub:25}
\begin{bmisc}[author]
\bauthor{\bsnm{Groeneboom},~\bfnm{Piet}\binits{P.}}
(\byear{2025}).
\btitle{{D}econvolution}.
\bhowpublished{\url{https://github.com/pietg/deconvolution}}.
\end{bmisc}
\endbibitem

\bibitem[\protect\citeauthoryear{Groeneboom and
  Jongbloed}{2003}]{GroJoUnifDec:03}
\begin{barticle}[author]
\bauthor{\bsnm{Groeneboom},~\bfnm{P.}\binits{P.}} \AND
  \bauthor{\bsnm{Jongbloed},~\bfnm{G.}\binits{G.}}
(\byear{2003}).
\btitle{Density estimation in the uniform deconvolution model}.
\bjournal{Statist. Neerlandica}
\bvolume{57}
\bpages{136--157}.
\bdoi{10.1111/1467-9574.00225}
\bmrnumber{2035863 (2004k:62121)}
\end{barticle}
\endbibitem

\bibitem[\protect\citeauthoryear{Groeneboom and
  Jongbloed}{2014}]{piet_geurt:14}
\begin{bbook}[author]
\bauthor{\bsnm{Groeneboom},~\bfnm{Piet}\binits{P.}} \AND
  \bauthor{\bsnm{Jongbloed},~\bfnm{Geurt}\binits{G.}}
(\byear{2014}).
\btitle{Nonparametric Estimation under Shape Constraints}.
\bpublisher{Cambridge Univ. Press}, \baddress{Cambridge}.
\end{bbook}
\endbibitem

\bibitem[\protect\citeauthoryear{Groeneboom and Wellner}{1992}]{GrWe:92}
\begin{bbook}[author]
\bauthor{\bsnm{Groeneboom},~\bfnm{P.}\binits{P.}} \AND
  \bauthor{\bsnm{Wellner},~\bfnm{J.~A.}\binits{J.~A.}}
(\byear{1992}).
\btitle{Information bounds and nonparametric maximum likelihood estimation}.
\bseries{DMV Seminar}
\bvolume{19}.
\bpublisher{Birkh\"auser Verlag}, \baddress{Basel}.
\bmrnumber{1180321 (94k:62056)}
\end{bbook}
\endbibitem

\bibitem[\protect\citeauthoryear{Groeneboom and Wellner}{2001}]{piet_jon:01}
\begin{barticle}[author]
\bauthor{\bsnm{Groeneboom},~\bfnm{P.}\binits{P.}} \AND
  \bauthor{\bsnm{Wellner},~\bfnm{J.~A.}\binits{J.~A.}}
(\byear{2001}).
\btitle{Computing {C}hernoff's distribution}.
\bjournal{J. Comput. Graph. Statist.}
\bvolume{10}
\bpages{388--400}.
\bdoi{10.1198/10618600152627997}
\bmrnumber{1939706}
\end{barticle}
\endbibitem

\bibitem[\protect\citeauthoryear{Jongbloed}{1998}]{geurt:98}
\begin{barticle}[author]
\bauthor{\bsnm{Jongbloed},~\bfnm{Geurt}\binits{G.}}
(\byear{1998}).
\btitle{The iterative convex minorant algorithm for nonparametric estimation}.
\bjournal{J. Comput. Graph. Statist.}
\bvolume{7}
\bpages{310--321}.
\bdoi{10.2307/1390706}
\bmrnumber{1646718}
\end{barticle}
\endbibitem

\bibitem[\protect\citeauthoryear{O'Sullivan and
  Roy~Choudhury}{2001}]{kingshuk:01}
\begin{barticle}[author]
\bauthor{\bsnm{O'Sullivan},~\bfnm{Finbarr}\binits{F.}} \AND
  \bauthor{\bsnm{Roy~Choudhury},~\bfnm{Kingshuk}\binits{K.}}
(\byear{2001}).
\btitle{An analysis of the role of positivity and mixture model constraints in
  {P}oisson deconvolution problems}.
\bjournal{J. Comput. Graph. Statist.}
\bvolume{10}
\bpages{673--696}.
\bdoi{10.1198/106186001317243395}
\bmrnumber{1938974}
\end{barticle}
\endbibitem

\bibitem[\protect\citeauthoryear{Reich et~al.}{2009}]{reich:09}
\begin{barticle}[author]
\bauthor{\bsnm{Reich},~\bfnm{Nicholas~G.}\binits{N.~G.}},
  \bauthor{\bsnm{Lessler},~\bfnm{Justin}\binits{J.}},
  \bauthor{\bsnm{Cummings},~\bfnm{Derek A.~T.}\binits{D.~A.~T.}} \AND
  \bauthor{\bsnm{Brookmeyer},~\bfnm{Ron}\binits{R.}}
(\byear{2009}).
\btitle{Estimating incubation period distributions with coarse data}.
\bjournal{Stat. Med.}
\bvolume{28}
\bpages{2769--2784}.
\bdoi{10.1002/sim.3659}
\bmrnumber{2750164}
\end{barticle}
\endbibitem

\bibitem[\protect\citeauthoryear{van~de Geer}{2000}]{geer:00}
\begin{bbook}[author]
\bauthor{\bparticle{van~de} \bsnm{Geer},~\bfnm{S.~A.}\binits{S.~A.}}
(\byear{2000}).
\btitle{Applications of empirical process theory}.
\bseries{Cambridge Series in Statistical and Probabilistic Mathematics}
\bvolume{6}.
\bpublisher{Cambridge University Press}, \baddress{Cambridge}.
\bmrnumber{1739079 (2001h:62002)}
\end{bbook}
\endbibitem

\end{thebibliography}

\end{document}